\documentclass[12pt,reqno]{amsart}

\usepackage{amsmath}
\usepackage{amsfonts}
\usepackage{amsthm}
\usepackage{amstext}
\usepackage{amssymb}
\usepackage{esint}
\usepackage{enumerate}
\usepackage{graphics}

\newcommand{\R}{\mathbb{R}}

\DeclareMathOperator*{\med}{median}
\DeclareMathOperator*{\midrange}{midrange}
\DeclareMathOperator*{\argmin}{arg\,min}

\theoremstyle{plain}
\newtheorem{lemma}{Lemma}[section]
\newtheorem{theorem}[lemma]{Theorem}
\newtheorem{corollary}[lemma]{Corollary}

\theoremstyle{definition}
\newtheorem{definition}[lemma]{Definition}

\begin{document}

\title[Formulas for homogeneous diffusion]{Statistical exponential formulas for homogeneous diffusion}

\author[M. Rudd]{Matthew Rudd}

\email{mbrudd@sewanee.edu}

\address{Sewanee: The University of the South \\
Sewanee, TN ~ 37383}

\date{\today}

\begin{abstract}
Let $\Delta^{1}_{p}$ denote the $1$-homogeneous $p$-Laplacian, for $1 \leq p \leq \infty$. 
This paper proves that the unique bounded, continuous viscosity solution $u$ of the Cauchy
problem 
\[
\left\{ \begin{array}{c}
u_{t} \ - \ ( \frac{p}{ \, N + p - 2 \, } ) \, \Delta^{1}_{p} u ~ = ~ 0 \quad \mbox{for} \quad x \in \R^{N} , \quad t > 0  \\
\\
u(\cdot,0) ~ = ~ u_{0} \in BUC( \R^{N} ) 
\end{array} \right.
\]
is given by the exponential formula
\[
u(t) ~ := ~ \lim_{n \to \infty}{ \left( M^{t/n}_{p} \right)^{n} u_{0} } \ ,
\]
where the statistical operator $M^{h}_{p} \colon BUC( \R^{N} ) \to BUC( \R^{N} )$ is defined by
\[
\left(M^{h}_{p} \varphi \right)(x) := 
(1-q) \med_{\partial B(x,\sqrt{2h})}{ \left\{ \, \varphi \, \right\} } + 
q \fint_{\partial B(x,\sqrt{2h})}{ \varphi \, ds } \ ,
\]
with $q := \frac{ N ( p - 1 ) }{ N + p - 2 }$, 
when $1 \leq p \leq 2$ and by
\[
\left(M^{h}_{p} \varphi \right)(x) := 
( 1 - q ) \midrange_{\partial B(x,\sqrt{2h})}{ \left\{ \, \varphi \, \right\} } + 
q \fint_{\partial B(x,\sqrt{2h})}{ \varphi \, ds  } \ ,
\]
with $q = \frac{ N }{ N + p - 2 }$, when $p \geq 2$. Possible extensions to problems with 
Dirichlet boundary conditions and to homogeneous diffusion on metric measure spaces are
mentioned briefly.
\end{abstract}

\maketitle

\section{Introduction} \label{intro}

\subsection{Linear diffusion and averaging}

Linear diffusion is the archetypal averaging process, and the most useful 
representation formulas for the solution of the heat equation convey its underlying statistics clearly.
The most famous such formula involves convolution with the heat kernel: 
given an admissible initial value $u_{0}$, the solution $u(x,t)$ of the Cauchy problem 
\begin{equation} \label{heat}
\left\{ \begin{array}{c}
u_{t} \, -  \, \Delta u ~ = ~ 0 \quad \mathrm{for} \quad x \in \R^{N} \quad \mbox{and} \quad \ t > 0 \, , \\
\\
u(x,0) ~ = ~ u_{0}(x) \quad \mbox{for} \quad x \in \R^{N} 
\end{array} \right.
\end{equation}
is simply given by
\begin{equation} \label{hk} 
u(x,t) ~ = ~ \int_{\R^{N}}{ P_{t}(x-y) u_{0}(y) \, dy } \, , 
\end{equation}
where, for $z \in \mathbb{R}^{N}$ and $t > 0$,  
\begin{equation} \label{hk2}
P_{t}(z) ~ := ~ \frac{1}{ \left( 4 \pi t \right)^{N/2} } \ \mbox{exp}\left( - \, \frac{ \, |z|^2 }{ \, 4t \, } \right) 
\end{equation}
is the Euclidean heat kernel. Formulas \eqref{hk} and \eqref{hk2} show that $u(x,t)$ is a weighted average of the initial 
data, with weights determined by a Gaussian with center $x$ and variance $2t$. The various properties of $u(x,t)$ when $t > 0$ 
follow directly from these formulas; it is clear, for example, that $u(x,t)$ is smooth
and that $u(x,t)$ is positive everywhere as long as $u_{0}$ is nonnegative and positive on a set of positive measure.

Heat kernel methods work beautifully in many different settings, and their development and deployment over 
the last thirty years have been particularly impressive (cf. \cite{grigoryan:hka09}, \cite{varopoulos:agg92}).  
Heat kernel techniques are fundamentally linear, however, and therefore do not apply to parabolic equations involving nonlinear operators. Semigroup methods, on the other hand, circumvent this dependence on linearity through their reliance 
on resolvents (\cite{crandall:gsn71},\cite{goldstein:slo85},\cite{pazy:slo83}). 
Following this approach, the solution of \eqref{heat} is given instead by the exponential formula
\begin{equation} \label{hille}
u(t) ~ = ~ \lim_{n \to \infty}{ \left( I - \frac{t}{n} \, \Delta \right)^{-n} u_{0} } \, .
\end{equation}
Despite the elegance of this formula and the theory behind it, \eqref{hille} does little to explain what actually 
happens to $u_{0}$ as it evolves according to \eqref{heat}; the infinite propagation speed mentioned above 
is certainly not obvious from \eqref{hille}, for example. Even worse, from the point of view of the present  
paper, is that verifying \eqref{hille} requires first analyzing the elliptic problems
\[
v \, - \, \lambda \, \Delta v ~ = ~ f \, , \quad \mbox{for given} \quad \lambda > 0 \quad \mbox{and} \quad f \, ,
\]
and obtaining careful estimates in appropriate spaces. Our philosophy is that we should
proceed the other way: we should first obtain a practical formula for the solution of a parabolic initial-value problem,
and we should then use that formula to gain insight into both the parabolic problem and the elliptic problems related to it.

Guided by this principle, this paper develops exponential formulas based not on resolvents, but on the local spatial 
statistics of the generator of the semigroup governing the evolution of the initial value $u_{0}$.  In the case of the heat 
equation, for instance, we can exploit the relationship given in Lemma \ref{taylor-mean} between the Laplacian and linear averaging, in which
we use the standard notation for integral average: for a measure $\mu$, a $\mu$-measurable set $E$, and a $\mu$-measurable
function $f$,
\[
\fint_{E}{ f \, d\mu } ~ := ~ \frac{1}{ \mu(E) } \, \int_{E}{ f \, d\mu } \ .
\]
We also adopt the usual notation $B(x,r)$ for an open ball with center $x$, radius $r \geq 0$, 
and spherical boundary $\partial B(x,r)$.

\begin{lemma} \label{taylor-mean}
For an open set $\Omega \subset \R^{N}$, $x \in \Omega$, and a smooth function 
$\varphi \colon \Omega \to \R$, 
\begin{equation} \label{taylor-mean-eqn}
\varphi(x) ~ - ~  \fint_{\partial B(x,\sqrt{2h})}{ \varphi(s) \, ds } ~ = ~ 
- \ \frac{h}{N} \, \Delta \varphi (x) ~ + ~ o( h ) \ . 
\end{equation}
\end{lemma}

This lemma follows directly from an elementary Taylor expansion, and identity \eqref{taylor-mean-eqn} is 
precisely the semigroup generation formula that we need. Combining it with a classical result like 
the Lax Equivalence Theorem \cite{lax:fa02} or Chernoff's 
Product Formula \cite{goldstein:slo85} yields 
an averaging representation for the solution $u$ of 
\begin{equation} \label{mod-heat}
\left\{ \begin{array}{c}
u_{t} \, -  \, \frac{1}{N} \, \Delta u ~ = ~ 0 \quad \mathrm{for} \quad x \in \R^{N} \quad \mbox{and} \quad t > 0 \, , \\
\\
u(x,0) ~ = ~ u_{0}(x) \quad \mbox{for} \quad x \in \R^{N} \ .
\end{array} \right.
\end{equation}
Specifically, let $BUC( \R^{N} )$ denote the space of bounded, uniformly continuous functions 
on $\R^{N}$ and, for a given $h > 0$, define the linear operator 
$M^{h}_{2} \colon BUC(\R^{N}) \to BUC(\R^{N})$ by
\[
\left( M^{h}_{2} \varphi \right)(x) ~ := ~ \fint_{\partial B(x,\sqrt{2h})}{ \varphi \, ds } \ , \quad \mbox{for} \quad 
x \in \R^{N} \ .
\]
We then have the following:
\begin{theorem} \label{chernoff-thm}
Let $u_{0} \in BUC( \R^{N} )$ be given. The function $u \colon [0, \infty) \to BUC( \R^{N} )$ 
defined by
\begin{equation} \label{chernoff}
u(t) ~ := ~ \lim_{n \to \infty}{ \left( M^{t/n}_{2} \right)^{n} u_{0} } 
\end{equation}
is the unique bounded, continuous solution of \eqref{mod-heat}.
\end{theorem}


As with the heat kernel representation \eqref{hk}, we can easily deduce the properties of the solution of
\eqref{mod-heat} from formula \eqref{chernoff}. The regularity of $u$ follows easily from \eqref{chernoff}, for example, thanks 
to the well-known smoothing effect of linear averaging. The infinite speed of propagation for the heat equation also 
follows easily: if $u_{0}$ is nonnegative and compactly supported, then the support of $u_{0}$ will expand a distance $\sqrt{ 2t/n }$
in all directions after an application of $M^{t/n}_{2}$, resulting in an expansion of $\sqrt{2nt}$ in all directions after $n$
iterations. Letting $n \to \infty$, it is clear that $u(x,t)$ will be supported on all of $\R^{N}$ for any $t > 0$. 
For future reference, note that this same calculation establishes an infinite speed of propagation for the
Cauchy problems considered in Section \ref{homo}.

When $N = 1$, the operator $M^{h}_{2}$ is particularly simple, since we then have
\begin{equation} \label{1d-avg}
\left( M^{h}_{2} \varphi \right)(x) ~ = ~ \frac{ \ \varphi( x - \sqrt{2h} ) + \varphi( x + \sqrt{2h} \ ) }{ 2 } \ \ .
\end{equation}
In this case, formula \eqref{chernoff} is a fully discrete forward Euler scheme for
the heat equation in one space dimension. Well-known arguments related to simple random walks 
on the line \cite{lawler:rwh10} then lead directly from \eqref{chernoff} to 
the heat kernel representation \eqref{hk}. In higher dimensions, the averaging operator $M^{h}_{2}$ 
no longer acts on discrete sets, but we can still combine the obvious discrete approximation 
of the average over a sphere with standard random walk techniques to derive \eqref{hk} from \eqref{chernoff}.

The exponential formula \eqref{chernoff} is thus clearly related to finite difference methods for the
heat equation, but it does not seem to have been used as a purely theoretical tool, despite its correspondence
with our common intuition about linear diffusion.  Furthermore, since our interest in \eqref{chernoff} is not 
driven by numerics, the computational limitations of forward Euler schemes are irrelevant here.  In fact, we
contend that formula \eqref{chernoff} is superior to other representations of the solution of \eqref{mod-heat}; 
formula \eqref{chernoff} exposes the averaging inherent in linear diffusion, permits an elementary analysis of solutions 
of \eqref{mod-heat}, can be modified easily to accommodate Dirichlet boundary conditions (see Section \ref{future} below), 
leads naturally to the mean value property of harmonic functions, 
and can be used to derive the heat kernel formula \eqref{hk}. 
Even more significant is the fact that straightforward variations of formula \eqref{chernoff} provide 
representations of the solutions of statistically-driven nonlinear diffusion equations. Substantiating this last 
claim is the goal of this paper.



\subsection{Overview of paper and background results}

As intimated above, this paper develops exponential formulas for the continuous viscosity solutions of Cauchy problems 
of the general form
\begin{equation} \label{gen}
\left\{ \begin{array}{c} 
\displaystyle{ u_{t} \, + \, F(Du, D^{2} u) ~ = ~ 0 \quad \mathrm{for} \quad x \in \R^{N} \quad \mbox{and} \quad t > 0 \, , } \\
\\
\displaystyle{ u(x,0) ~ = ~ u_{0}(x) \quad \mbox{for} \quad x \in \R^{N} \ , }
\end{array} \right.
\end{equation}
where $u_{0} \in BUC(\R^{N})$, $Du$ denotes the spatial gradient of the real-valued function $u$, $D^2 u$ denotes its spatial Hessian, and 
$F = F(p,X)$ satisfies the conditions listed below; as usual, $S^{N}$ denotes the space of $N \times N$ real symmetric matrices with its 
standard partial ordering.
\smallskip
\begin{enumerate}[(i)]
\item
$F \, \colon \, \R^{N} \setminus \{ 0 \} \times S^{N} \to \R$ \ is continuous; excluding the case $p = 0$ 
accommodates singular gradient dependence in \eqref{gen}.
\smallskip
\item 
$F$ is $1$-homogeneous:
given $p \in \R^{N} \setminus \{ 0 \}$, $X \in S^{N}$, and $\alpha \in \R$,
\[
F( \alpha p, \alpha X ) ~ = ~ \alpha F( p, X ) .
\]
\item
$F$ is degenerate elliptic: 
for a given $p \in \R^{N} \setminus \{ 0 \}$, 
\[
F( p, X ) \leq F( p, Y ) \quad \mbox{when} \quad Y \leq X .
\]
\end{enumerate}

Since representation formulas for solutions are the focus of this paper, we presume a 
familiarity with viscosity solutions and refer to other works
for definitions, existence proofs, comparison principles, and so on. The Users' 
Guide \cite{crandall:ugv92} and the lecture notes \cite{bardi:vsa97} are basic references, 
several papers by Juutinen, Kawohl and their coauthors
(\cite{akagi:euv08}, \cite{julin:npe12}, \cite{juutinen:evs01}, \cite{juutinen:egi06}, \cite{kawohl:cpl00}, \cite{kawohl:cpv07}) 
address important issues related to comparison principles and the definitions of solutions, and 
the fundamental paper \cite{giga:cpc91} by Giga, Goto, Ishii 
and Sato is an essential reference (see also \cite{giga:see06}), as it contains the 
comparison principle on which our uniqueness statements depend. Also, the results that follow 
all depend on the modern viscosity version of Chernoff's Product Formula developed by Barles and 
Souganidis in \cite{barles:cas91}.

To summarize what follows, Section \ref{mc} shows that replacing the linear average in \eqref{chernoff} with the median yields
an exponential formula for motion by mean curvature; as discussed there, the median is 
a nonlinear average that must be handled a bit carefully. It is easier to work with the midrange, which averages the infimum and
supremum of a function over a set, and Section \ref{midrange} proves that iterating this particular nonlinear average
provides a formula for the solution of the Cauchy problem for the parabolic infinity-Laplacian. Section \ref{homo} 
combines the results from Sections \ref{intro}, \ref{mc} and \ref{midrange} to establish exponential
formulas for the Cauchy problems involving the parabolic $1$-homogeneous $p$-Laplacian for $1 \leq p \leq \infty$. (The $1$-homogeneous $p$-Laplacian is 
also known as the normalized or game-theoretic $p$-Laplacian.) Finally, Section \ref{future} 
speculates on possible generalizations of our results to parabolic problems with Dirichlet boundary conditions 
and to homogeneous diffusion on metric measure spaces. 

All of the exponential formulas proven below are the same when $N=1$, since the averaging operators used all reduce 
to \eqref{1d-avg} in that case. We therefore assume henceforth that $N \geq 2$. 

\section{Medians and mean curvature flow in $\R^{N}$} \label{mc}

This section establishes a statistical exponential formula for the solution $u$ of the 
level set formulation of mean curvature flow, 
\begin{equation} \label{mc-flow}
\left\{ \begin{array}{c} 
\displaystyle{ u_{t} \ - \ |Du| \ \mbox{div}\left( \frac{Du}{ \, | Du | \, } \right)  ~ = ~ 0 , 
\quad \mbox{for} \quad x \in \R^{N}, \ \ t > 0 , } \\
\\
\displaystyle{ u(x,0) ~ = ~ u_{0}(x) , \quad \mbox{for} \quad x \in \R^{N} \ . }
\end{array} \right.
\end{equation}
This well-known front propagation model was introduced by Osher and Sethian \cite{osher:fpc88} and has since been 
studied extensively by many authors. Our exponential formula for its solution is a simple nonlinear analogue of 
the averaging representation \eqref{chernoff} for the solution $u$ of the heat equation \eqref{mod-heat}; to provide a context for 
the formula, we briefly review some particularly relevant earlier work.

Since \eqref{mc-flow} appeared in \cite{osher:fpc88}, developing efficient and provably convergent algorithms for 
approximating its solution has been a fundamental problem. In \cite{merriman:mmf94}, 
Bence, Merriman and Osher proposed their famous algorithm for mean curvature motion; during each iteration, their algorithm 
solves the heat equation with appropriate initial data and then thresholds the resulting solution. Their algorithm's  
performance motivated much work on its convergence properties, resulting in the basic references \cite{barles:spc95}, 
\cite{evans:cam93}, and \cite{ishii:tdt99}. 
Of these three papers, the work by Evans \cite{evans:cam93} is perhaps most closely related to the present paper, 
as it applies the resolvent-based nonlinear semigroup machinery of Crandall and Liggett \cite{crandall:gsn71} to prove 
convergence of the Bence-Merriman-Osher scheme. 

Shortly thereafter, Catt\'{e}, Dibos and Koepfler \cite{catte:msm95} established a different
Crandall-Liggett exponential formula for the solution of \eqref{mc-flow} when $N=2$, basing their development on an axiomatic 
approach to image processing \cite{alvarez:afe93}. Specifically, Catt\'{e} et al. proved that the solution $u$ of \eqref{mc-flow} 
has the representation
\begin{equation} \label{catte}
u(t) ~ = ~ \lim_{n \to \infty}{ C_{t/n}^{n} u_{0} } \, ,
\end{equation}
where 
\[
C_{h} \varphi ~ := ~ \frac{1}{2} \left\{ S_{2h} \varphi + I_{2h} \varphi \right\} \, , 
\]
\[
\left( I_{h} \varphi \right)(x) ~ := ~ \inf_{\theta \in [0,\pi)}{ \sup_{x + \sigma(\theta,h)}{ \left\{ \, \varphi(y) \, \right\} } } \, ,
\]
\[
\left( S_{h} \varphi \right)(x) ~ := ~ \sup_{\theta \in [0,\pi)}{ \inf_{x + \sigma(\theta,h)}{ \left\{ \, \varphi(y) \, \right\} } } \, ,
\]
and $\sigma(\theta,h)$ denotes a segment centered at $0$ with direction $\theta$ and length $2 \sqrt{2h}$.

Although \cite{catte:msm95} appeared roughly twenty years ago, it has become better known through a more
recent work by Kohn and Serfaty \cite{kohn:dcb06} that produced a variant of formula \eqref{catte} from a
very different perspective. Kohn and Serfaty arrived at their version of \eqref{catte} by way of the dynamic programming
principle for a simple deterministic two-player game in $\R^{2}$ that we review briefly. 
During each round of the game, the first player chooses a direction 
$v \in S^{1}$ that the second player either accepts or reverses, thereby determining the
direction $w = \pm v$; the game position then moves 
a distance $\sqrt{2h}$ in the direction $w$. When the game ends,  
the first player pays the second player the amount $u_{0}( x_{T} )$, where $x_{T}$ is the game position when
the game ends at time $T$. 
Since the first player's value function $u_{h}$ corresponds to playing optimally, $u_{h}$ satisfies the 
dynamic programming principle
\begin{equation} \label{dpp}
u_{h}( x, kh) ~ = ~ \min_{ \mathbf{v} \in S^{1} }{ \max_{ b = \pm 1 }{ \left\{ \, u_{h}( x + \sqrt{2h} bv, (k+1) h ) \, \right\} } } , 
\end{equation}
with $u_{h}(x,T) = u_{0}(x)$. Using \eqref{dpp}, Kohn and Serfaty proved that the value functions 
$u_{h}$ converge, as $h \to 0$, to the solution of \eqref{mc-flow} (after a simple change of variables to go from
this terminal-value game to the initial-value formulation above). Equation \eqref{dpp} 
and the results of \cite{kohn:dcb06} therefore show that we can approximate the mean curvature motion of a curve, over a small
time step $h$, by tracking the midpoint of a segment of length $2\sqrt{2h}$ as its endpoints traverse the curve; 
iterating this procedure yields an exponential formula for the solution of \eqref{mc-flow}.


Ruuth and Merriman presented a similar description of mean curvature flow, from yet another viewpoint, 
in \cite{ruuth:cgm00}, an interesting paper that does not seem to have received sufficient attention. 
They proved that, for a given closed curve $\gamma$ in the plane, one can approximate the mean curvature flow of
$\gamma$ for a small time step $h$ by tracking the center of a circle of radius $\sqrt{2h}$ as it traverses $\gamma$
in such a way that exactly half of its area is always inside $\gamma$. Note that the approximation scheme corresponding 
to \eqref{dpp}, on the other hand, tracks the center of a circle as it traverses $\gamma$ in such a way that half of 
its circumference is always inside $\gamma$. 

While the paper by Kohn and Serfaty was making the rounds as a preprint, Oberman published his paper 
on provably convergent median schemes for mean curvature flow \cite{oberman:cmd04}. In it, he presented a 
forward Euler method for approximating the solution of \eqref{mc-flow} on a rectangular grid; at each iteration, the 
algorithm updates each grid value with the median of its neighboring grid values, with neighbors defined carefully so as 
to reduce errors from poor angular resolution. In particular, using a small grid spacing and a wide computational stencil 
enables the selection of neighbors that approximate a circle centered at the grid point of interest.

All of these papers have suggested connections between mean curvature flow and appropriately interpreted median operators, with 
Oberman's work doing so explicitly in a discrete context. To develop these connections further, we summarize the properties of medians
of measurable and continuous functions, beginning with the definition and proceeding to some 
recent results that should be of independent interest. 

\begin{definition}[\cite{ziemer:wdf89}]
If $u \colon E \to \R$ is measurable and $0< |E| < \infty$, then $m$ is a \textit{median} of $u$ over $E$ if and only if
\[
| \,  \{ \, u < m \} \, | ~ \leq ~ \frac{1}{2} \, | \, E \, | \quad \mbox{and} \quad 
| \,  \{ \, u > m \} \, | ~ \leq ~ \frac{1}{2} \, | \, E \, |  \ .
\]
\end{definition}

We denote the set of all medians of $u$ over $E$ by $\displaystyle{ \med_{E}{ \{ u \} } } $. 
As discussed in \cite{ziemer:wdf89}, $\displaystyle{ \med_{E}{ \{ u \} } }$ is a non-empty compact interval, and
\begin{equation} \label{homo-trans}
\med_{E}{ \left\{ \alpha u + \beta \right\} } ~ = ~ \alpha \med_{E}{ \{ u \} } + \beta 
\end{equation}
for any constants $\alpha$ and $\beta$. We will exploit these homogeneity and translation invariance properties below, 
as well as the obvious stability of medians: if $u \in L^{\infty}(E)$, then 
$m \leq \| u \|_{\infty}$ for any $\displaystyle{ m \in \med_{E}{ \{ u \} } } $. 

If $u$ is merely measurable, the set of medians of $u$ over $E$ can clearly have positive Lebesgue measure. 
Continuous functions, on the other hand, have unique medians over compact connected sets, 
as shown by the following results from \cite{hartenstine:sfe13}.
These facts pave the way for further analysis and explain our restriction to continuous initial data throughout this paper. 
In the statement of Lemma \ref{fundamental}, $LSC(E)$, $USC(E)$ and $C(E)$ denote, respectively, the lower 
semicontinuous, upper semicontinuous, and continuous real-valued functions on the measurable set $E$. 

\begin{lemma}[\cite{hartenstine:sfe13}] \label{fundamental}
Suppose that $L \in LSC(E)$, $U \in USC(E)$, $L \geq U$, and that $E \subset \R^{k}$ is compact and connected.
If $\displaystyle{ m \in \med_{E}{L} } $ and $\displaystyle{M \in \med_{E}{U} }$, then $m \geq M$.
\end{lemma}

\begin{corollary}[\cite{hartenstine:sfe13}] \label{uni-mono}
Suppose that $E \subset \R^{k}$ is compact and connected.
If $v \in C(E)$, then the median of $v$ over $E$ is unique. Moreover, if 
$u \in C(E)$ satisfies $u \geq v$, then $ \displaystyle{ \med_{E}{u} \geq \med_{E}{v} } $.
\end{corollary}

Lemma \ref{taylor-mean} is the essential ingredient in the proof of the exponential formula
\eqref{chernoff} for the solution of \eqref{mod-heat}; it shows that repeated linear averaging provides 
a consistent approximation scheme for linear diffusion. The next result shows that the median, a nonlinear average,   
plays a corresponding role for mean curvature flow.  It is new in dimensions $N > 2$, having been established
in \cite{hartenstine:asc11} for functions of two variables. Henceforth, it will be 
convenient to let $\Delta_{1}$ denote the $1$-Laplacian operator, the elliptic part of 
equation \eqref{mc-flow}; $\Delta_{1}$ is thus defined formally by
\begin{equation} \label{1-Laplacian} 
\Delta_{1} \varphi ~ := ~ |D\varphi| \ \mbox{div}\left( \frac{D\varphi}{ \, | D\varphi | \, } \right) \ .
\end{equation}
It will also be useful below to recall that, when 
$D \varphi(x) \neq 0$, 
\begin{equation} \label{1Lap-trace}
\Delta_{1}\varphi(x) ~ = ~ \mbox{tr}\left( \left. D^{2}\varphi(x) \right\vert_{\Sigma} \, \right)  \ , 
\end{equation}
where $\Sigma$ is the plane orthogonal to $D \varphi(x)$; this description of the $1$-Laplacian 
plays a key role in the related papers \cite{kawohl:vpl11} and \cite{kawohl:snp12}.

\begin{lemma} \label{taylor-median}
Let $\Omega \subset \R^{N}$ be open, $N \geq 2$. For $x \in \Omega$ and a smooth function 
$\varphi \colon \Omega \to \R$ with $|D \varphi(x)| \neq 0$, 
\begin{equation} \label{1LapSpheres}
\varphi(x) ~ - ~ \med_{\partial B(x,\sqrt{2h})}{ \left\{ \, \varphi \, \right\} } ~ = ~ 
- \ \frac{ h }{ N-1 } \ \Delta_{1} \varphi(x) ~ + ~ 
o( h ) \ .
\end{equation}

\end{lemma}

\begin{proof}

We begin by recalling the proof given in \cite{hartenstine:asc11} when $N=2$.
Since $|D \varphi(x)| \neq 0$, the Implicit Function Theorem guarantees that, for sufficiently 
small $h > 0$, the level sets of $\varphi$ are smooth curves that foliate the closure of the ball $ B(x,\sqrt{2h}) $.
There is a unique level curve corresponding to the median of $\varphi$ over the circle $\partial B(x,\sqrt{2h}) $; 
by definition of the median, this curve must separate the circle into two arcs of equal length, yielding
antipodal points $y^{+}_{h}, y^{-}_{h} \in \partial B(x,\sqrt{2h})$ such that
\begin{equation} \label{antipode2}
\varphi(y^{+}_{h}) ~ = ~ \varphi( y^{-}_{h} ) ~ = ~ \med_{\partial B(x,\sqrt{2h})}{ \{ \, \varphi \, \} } \ .
\end{equation}

Letting
\[
v_{h} ~ :=  ~ \frac{ \ y^{+}_{h} - x \ }{ \sqrt{2h \, } } ~ =  ~ - ~ \frac{ \ y^{-}_{h} - x \ }{ \sqrt{2h \, } } \quad ,
\]
we have the Taylor expansions 
\[
\varphi(y^{+}_{h}) ~ = ~ 
\varphi(x) + \sqrt{2h} \, D\varphi(x)\cdot v_{h} + h \, D^{2}\varphi(x) v_{h} \cdot v_{h} 
+ o( h ) 
\]
and
\[
\varphi(y^{-}_{h}) ~ = ~ 
\varphi(x) - \sqrt{2h} \, D\varphi(x)\cdot v_{h} + h \, D^{2}\varphi(x) v_{h} \cdot v_{h}
+ o( h ) \ ,
\]
which must be equal by \eqref{antipode2}. Subtracting one from the other, we see that
\[
2\sqrt{2h} \ D\varphi(x) \cdot v_{h} ~ = ~ o(h) \ ,
\]
from which it follows that
\begin{equation} \label{vh}
v_{h} ~ = ~ \frac{ \ D\varphi(x)^{\perp} \ }{ \ | D \varphi(x) | \ } ~  + ~ e_{h} \ , 
\end{equation}
where
\[
D \varphi(x)^{\perp} ~ = ~ \left[ \begin{array}{lr} 0 & -1 \\ 1 & 0 \end{array} \right] D\varphi(x) 
\quad \quad \mbox{and} \quad \quad
2\sqrt{2h} \ D \varphi(x) \cdot e_{h} ~ = ~ o(h) \ .
\]
Substituting the expression for $v_{h}$ from \eqref{vh} into either Taylor expansion and calculating directly 
completes the proof in this case.

Before proceeding to the proof in higher dimensions, we rephrase the proof just given in a form that will be 
easier to generalize. Using a normal coordinate system $(y_{1},y_{2})$ centered at $x$, we can assume that $x = (0,0)$ and that
$D \varphi(x) $ determines the direction of the $y_{2}$-axis. Up to second order (i.e., modulo $o(h)$ errors since the radius is 
$\sqrt{2h}$), we have the following: 
the level curve passing through $x$ is 
the parabola $y_{2} = - \frac{1}{2} \kappa y_{1}^{2} $ (where $\kappa$ is the curvature at the vertex $x$), 
the level curves of $\varphi$ in $B(x,\sqrt{2h})$ are translates of this parabola, and the median
corresponds to the level curve $ y_{2} = \kappa h - \frac{1}{2} \kappa y_{1}^{2} $. In this coordinate system, then,
\[
\varphi( \pm \sqrt{2h}, 0 ) ~ = ~ \med_{\partial B(x,\sqrt{2h})}{ \{ \, \varphi \, \} } + o(h) \ ,
\]
after which the rest of the proof remains the same.

When $ N = 3 $, we know by the Implicit Function Theorem that, for sufficiently 
small $h > 0$, the level sets of $\varphi$ are smooth $2$-dimensional surfaces that foliate the closure of 
the ball $ B(x,\sqrt{2h}) $. Introduce coordinates $ ( y_{1}, y_{2}, y_{3} ) $ such that $x = (0,0,0)$ and 
the $y_{3}$-axis is parallel to $D \varphi(x) $; up to second order, the level sets of $\varphi$ in $ B(x,\sqrt{2h}) $
are translates of a quadratic polynomial $ p( y_{1}, y_{2} ) $ (namely, the second fundamental form of the level surface through $x$), 
and the median corresponds to the level surface
$ y_{3} = \alpha + p( y_{1}, y_{2} ) $ whose intersection $\Gamma$ with $ \partial B(x,\sqrt{2h}) $ bisects the surface area of 
$ \partial B(x,\sqrt{2h}) $. Consequently, since $p( y_{1}, y_{2} ) =  p( -y_{1}, -y_{2} )$, $\Gamma$ must intersect the plane
$\{ y_{3} = 0 \} $ in at least two distinct pairs of antipodal points. These pairs of antipodal points yield independent 
unit vectors $v_{h}$ and $w_{h}$ such that
\[
\med_{\partial B(x,\sqrt{2h})}{ \{ \, \varphi \, \} } ~ = ~
\varphi(x) \pm \sqrt{2h} \, D\varphi(x)\cdot v_{h} + h \, D^{2}\varphi(x) v_{h} \cdot v_{h} 
+ o( h ) 
\]
and 
\[
\med_{\partial B(x,\sqrt{2h})}{ \{ \, \varphi \, \} } ~ = ~
\varphi(x) \pm \sqrt{2h} \, D\varphi(x)\cdot w_{h} + h \, D^{2}\varphi(x) w_{h} \cdot w_{h} 
+ o( h ) \ ,
\]
from which we conclude as above that $v_{h}$ and $w_{h}$ are nearly orthogonal to $D \varphi(x) $. 
More precisely, we have
\begin{equation} \label{3d-v}
\med_{\partial B(x,\sqrt{2h})}{ \{ \, \varphi \, \} } ~ = ~
\varphi(x) + h \, D^{2}\varphi(x) v_{h} \cdot v_{h} + o( h ) 
\end{equation}
and 
\begin{equation} \label{3d-w}
\med_{\partial B(x,\sqrt{2h})}{ \{ \, \varphi \, \} } ~ = ~
\varphi(x) + h \, D^{2}\varphi(x) w_{h} \cdot w_{h} + o( h ) \ ,
\end{equation}
and averaging equations \eqref{3d-v} and \eqref{3d-w} yields
\begin{equation} \label{3d-v-and-w}
\med_{\partial B(x,\sqrt{2h})}{ \{ \, \varphi \, \} } =
\varphi(x) + \frac{ h }{ 2 } \left( D^{2}\varphi(x) v_{h} \cdot v_{h} + D^{2}\varphi(x) w_{h} \cdot w_{h} \right) + o( h ) .
\end{equation} 
Using the characterization \eqref{1Lap-trace} of the $1$-Laplacian above, we can rewrite \eqref{3d-v-and-w} in the form
\[
\med_{\partial B(x,\sqrt{2h})}{ \{ \, \varphi \, \} } ~ = ~
\varphi(x) ~ + ~ \frac{ h }{ 2 } \, \Delta_{1}\varphi(x) + o( h ) \ ,
\]
which is equation \eqref{1LapSpheres} when $N=3$.

A similar argument now applies in higher dimensions, using a quadratic polynomial of $(N-1)$ variables to approximate the
smooth submanifold on which $\varphi$ achieves its median. 

\end{proof}

Given $h > 0$ and $\varphi \in BUC( \R^{N} )$, we define the 
nonlinear averaging operator $M^{h}_{1} \colon BUC( \R^{N} ) \to BUC( \R^{N} )$ by
\begin{equation} \label{M1h}
\left( M^{h}_{1} \varphi \right)(x) ~ := ~ \med_{\partial B(x,\sqrt{2h})}{ \left\{ \, \varphi \, \right\} } \ , \quad \mbox{for} \quad x \in \R^{N} \ .
\end{equation}
The fact that $M^{h}_{1} \varphi$ is bounded when $\varphi$ is bounded follows directly from the stability 
of the median mentioned above, and Proposition 2.2 of \cite{hartenstine:sfe13} verifies that $M^{h}_{1} \varphi$ is
uniformly continuous whenever $\varphi$ is uniformly continuous. We see from the preceding results that the 
operator $M^{h}_{1}$ has the following important properties: 
\begin{enumerate}[(i)]

\item
translation invariance: $M^{h}_{1} \left( v + c \right) ~ = ~ M^{h}_{1}v + c$ for any $v \in BUC( \R^{N} )$  and $c \in \R$.  

\item
monotonicity: $M^{h}_{1} v ~ \leq ~ M^{h}_{1} w$ ~ whenever $v, w \in BUC( \R^{N} )$ satisfy $v ~ \leq ~ w$. 

\item
$1$-homogeneity: if $\alpha \in \R$ and $v \in BUC( \R^{N} )$, then 
$M^{h}_{1}( \, \alpha v \, ) = \alpha \, M^{h}_{1}(v) $ .

\item
stability: $ \| M^{h}_{1} v \|_{\infty} \leq \| v \|_{\infty} $ for any $v \in BUC( \R^{N} )$.

\item
consistency: for any smooth ~ $\varphi$ with nonvanishing gradient, 
\[
\lim_{h \rightarrow 0}{ \left( \frac{ \varphi - \left( M^{h}_{1} \varphi \right) }{ h } \right) } ~ = ~ 
- \ \frac{ 1 }{ \, N-1 \, } \ \Delta_{1} \varphi \ .
\] 

\end{enumerate}
Consequently, a direct application of the framework developed in \cite{barles:cas91} proves the following:

\begin{theorem} \label{mc-soln-thm}
Let $u_{0} \in BUC(\R^{N})$ be given. The function $u \colon [0, \infty) \to BUC( \R^{N} )$ 
defined by
\begin{equation} \label{mc-soln}
u(t) ~ := ~ \lim_{n \to \infty}{ \left( M^{t/n}_{1} \right)^{n} u_{0} } 
\end{equation}
is the unique viscosity solution of 
\begin{equation} \label{mc-med-flow}
\left\{ \begin{array}{c}
\displaystyle{ u_{t} \ - \ \frac{ 1 }{ N-1 } \, \Delta_{1} u ~ = ~ 0 \quad \mathrm{for} \quad x \in \R^{N}, \  t > 0 \ , } \\
\\
u(x,0) ~ = ~ u_{0}(x) \quad \mbox{for} \quad x \in \R^{N} \ .
\end{array} \right.
\end{equation}
\end{theorem}

\noindent As mentioned toward the end of Section \ref{intro}, the uniqueness of this solution follows from the 
relevant comparison principle (\cite{giga:see06}, \cite{giga:cpc91}).

Although we only stated the exponential formula \eqref{chernoff} for the solution of the heat equation in terms
of averages over spheres, an analogous formula based on averages over balls certainly holds. It is not 
clear whether a version of Theorem \ref{mc-soln-thm} based on medians over closed balls is true, as 
we do not have a proof of the consistency result in Lemma \ref{taylor-median} when we replace spheres 
with closed balls. The earlier results of Bence-Merriman-Osher \cite{merriman:mmf94} and Ruuth-Merriman \cite{ruuth:cgm00}, 
however, suggest that using medians over closed balls should work; this would be an interesting issue to resolve.

\section{Midranges and the parabolic $\infty$-Laplacian} \label{midrange}

We now consider the Cauchy problem
\begin{equation} \label{infty-lap-cauchy}
\left\{ \begin{array}{c}
u_{t} \, - \, \Delta_{\infty} u ~ = ~ 0 \quad \mathrm{for} \quad x \in \R^{N} \quad \mbox{and} \quad t > 0 \, , \\
\\
u(x,0) ~ = ~ u_{0}(x) \quad \mbox{for} \quad x \in \R^{N} \ , 
\end{array} \right.
\end{equation} 
where $u_{0} \in BUC(\R^{N})$ and the $\infty$-Laplacian $\Delta_{\infty}$ is defined by
\[
\Delta_{\infty} \varphi ~ := ~ 
\frac{1}{ |D\varphi|^{2} } \, \sum_{i,j=1}^{N}{ \frac{\partial \varphi}{\partial x_{i}} \, \frac{\partial \varphi}{\partial x_{j}} \, \frac{\partial^{2} \varphi}{ \partial x_{i} \partial x_{j} } }
\]
for smooth $\varphi$ with $|D\varphi| \neq 0$. 
In contrast to the heat equation and the mean curvature equation, there are surprisingly few papers devoted to
parabolic problems involving the $\infty$-Laplacian. In \cite{juutinen:egi06}, Juutinen and Kawohl answered 
basic existence and uniqueness questions for the Cauchy problem \eqref{infty-lap-cauchy} as well as its analogue 
on bounded domains with Dirichlet boundary conditions; they proved, in particular, that \eqref{infty-lap-cauchy} has 
a unique bounded solution. Subsequently, Akagi, Juutinen and Kajikiya studied the asymptotic behavior of the 
solutions of these evolution problems, proving, among other results, that the optimal decay rate (in $L^{\infty}( \R^{N} )$) 
of the solution of \eqref{infty-lap-cauchy} when $u_{0}$ has compact support is $(t+1)^{-1/6}$. More recently,
Manfredi, Parviainen and Rossi \cite{manfredi:amv10} proved an asymptotic statistical characterization of 
solutions of \eqref{infty-lap-cauchy} that complements the present work.

To establish a statistical exponential formula for the solution of \eqref{infty-lap-cauchy},  
we recall the following relationship between the $\infty$--Laplacian and the average of extreme values:
\begin{lemma} \label{taylor-midrange}
For an open set $\Omega \subset \R^{N}$, $x \in \Omega$, and a smooth function 
$\varphi \colon \Omega \to \R$ with $|D \varphi(x)| \neq 0$, 
\begin{equation} \label{InftyLapBalls}
\varphi(x) ~ - ~ 
\midrange_{\partial B(x,\sqrt{2h})}{ \left\{ \, \varphi \, \right\} } ~ = ~ - \ h \, \Delta_{\infty} \varphi(x) ~ + ~ o( h ) \, ,
\end{equation}
where, for any compact set $K \subset \R^{N}$, 
\[
\midrange_{K}{ \left\{ \, \varphi \, \right\} } ~ := ~
\frac{1}{2} \left( \max_{ K }{ \left\{ \, \varphi \, \right\} } + \min_{ K }{  \left\{ \, \varphi \, \right\} } \right) \ .
\]
\end{lemma}
\noindent Identity \eqref{InftyLapBalls} has been used in various forms 
elsewhere (\cite{kawohl:snp12}, \cite{legruyer:he98}, \cite{legruyer:aml07}, \cite{manfredi:amv10}, \cite{oberman:cds05})
and follows from the elementary fact that the gradient is the direction of steepest ascent.
Note that, like the median operator discussed in the previous 
section, the midrange operator computes a nonlinear average and is monotone, stable, 
translation invariant (in the sense used earlier), and $1$-homogeneous. Moreover, 
Lemma \ref{taylor-midrange} shows that the midrange operator provides a consistent approximation of the 
$\infty$-Laplacian. Proceeding as in Section \ref{mc}, we therefore define the
nonlinear averaging operator $M^{h}_{\infty} \colon BUC(\R^{N}) \to BUC(\R^{N})$ by
\begin{equation} \label{M-h-infty}
\left( M^{h}_{\infty} \varphi \right)(x) ~ := ~ 
\midrange_{\partial B(x,\sqrt{2h})}{ \left\{ \, \varphi \, \right\} } \ , \quad \mbox{for} \quad x \in \R^{N} ,
\end{equation}
and apply the machinery of \cite{barles:cas91} to obtain
\begin{theorem}
Let $u_{0} \in BUC(\R^{N})$ be given. The function $u \colon [0, \infty) \to BUC( \R^{N} )$ 
defined by
\begin{equation} \label{infty-lap-soln}
u(t) ~ := ~ \lim_{n \to \infty}{ \left( M^{t/n}_{\infty} \right)^{n} u_{0} } 
\end{equation}
is the unique bounded, continuous viscosity solution of \eqref{infty-lap-cauchy}.
\end{theorem}

\section{Homogeneous diffusion in $\R^{N}$} \label{homo}

The evolution equations studied in Sections \ref{intro}, \ref{mc} and \ref{midrange} 
belong to the one-parameter family of Cauchy problems  
\begin{equation} \label{homo-cauchy}
\left\{ \begin{array}{c}
\displaystyle{ u_{t} \ - \ c_{p,N} \, \Delta^{1}_{p} u ~ = ~ 0 \quad \mbox{for} \quad x \in \R^{N} \quad 
\mbox{and} \quad t > 0 \, , } \\
\\
u(x,0) ~ = ~ u_{0}(x) \quad \mbox{for} \quad x \in \R^{N} \ ,
\end{array} \right.
\end{equation}
where 
\[
c_{p,N} ~ :=  ~ \frac{p}{ N + p - 2 } 
\]
and the $1$-homogeneous $p$-Laplacian $\Delta^{1}_{p}$ is defined, for $ 1 \leq p \leq \infty$, by
\begin{equation} \label{homo-pLap}
\Delta_{p}^{1} \varphi ~ := ~ \left\{ \begin{array}{ccc}
(1 - \frac{1}{p}) \Delta \varphi \, + \, (\frac{2}{p} - 1) \Delta_{1} \varphi \ & \mbox{if} & 1 \leq p \leq 2 \, , \\
& & \\
\frac{1}{p} \Delta \varphi \, + \, (1 - \frac{2}{p}) \Delta_{\infty} \varphi  & \mbox{if} & p \geq 2 \, . \\
\end{array} \right.
\end{equation}
Note that $\Delta^{1}_{1} = \Delta_{1}$ and $\Delta^{1}_{\infty} = \Delta_{\infty}$, while 
$\Delta^{1}_{2} = \frac{1}{2} \Delta$. 

Given $h > 0$ and $p \in [1,\infty]$, we define the statistical  
operator $M^{h}_{p} \colon BUC(\R^{N}) \to BUC(\R^{N})$ by
\[
\left(M^{h}_{p} \varphi \right)(x) := \left\{ \begin{array}{c}
\displaystyle{ (1-q) \med_{\partial B(x,\sqrt{2h})}{ \left\{ \, \varphi \, \right\} } + 
q \fint_{\partial B(x,\sqrt{2h})}{ \varphi \, ds } } \, , \quad 1 \leq p \leq 2 , \\
\\
\displaystyle{ ( 1 - q ) \midrange_{\partial B(x,\sqrt{2h})}{ \left\{ \, \varphi \, \right\} } + 
q \fint_{\partial B(x,\sqrt{2h})}{ \varphi \, ds  } } \, , \quad p \geq 2 ,
\end{array} \right.
\]
where
\begin{equation} \label{q}
q ~ = ~ q(p,N) ~ := ~ \left\{ \begin{array}{ccc}
\displaystyle{ \frac{ N ( p - 1 ) }{ N + p - 2 } \ , } & \mbox{if} & 1 \leq p \leq 2 \, , \\
& \\
\displaystyle{ \frac{ N }{ N + p - 2 } \ , } & \mbox{if} & p \geq 2 \, .
\end{array} \right.
\end{equation}
Since the operator $M^{h}_{p}$ is a simple linear combination of the averaging operators studied earlier, 
it clearly enjoys the same properties of homogeneity, stability, monotonicity, and translation invariance, with 
the consistency condition
\begin{equation} \label{homo-consistency}
\lim_{h \rightarrow 0}{ \left( \frac{ \varphi - \left( M^{h}_{p} \varphi \right) }{ h } \right) } ~ = ~ 
- \, c_{p,N} \, \Delta^{1}_{p} \varphi 
\end{equation}
following from Lemmas \ref{taylor-mean}, \ref{taylor-median}, and \ref{taylor-midrange}.
As in Sections \ref{mc} and \ref{midrange}, we can therefore combine the results of \cite{barles:cas91} 
and the comparison principle from \cite{giga:cpc91} 
to establish
\begin{theorem} 
Let $u_{0} \in BUC(\R^{N})$ be given. The function $u \colon [0, \infty) \to BUC( \R^{N} )$ 
defined by
\begin{equation} \label{homo-cauchy-soln}
u(t) ~ := ~ \lim_{n \to \infty}{ \left( M^{t/n}_{p} \right)^{n} u_{0} } 
\end{equation}
is the unique bounded, continuous viscosity solution of \eqref{homo-cauchy}.
\end{theorem}

As remarked after the statement of Theorem \ref{chernoff-thm}, the infinite speed of propagation of the
heat equation follows directly from formula \eqref{chernoff}. The exact same argument shows that problem \eqref{homo-cauchy}
also exhibits an infinite speed of propagation as long as the coefficient $q$ is nonzero in the definition of $M^{h}_{p}$; 
$q \neq 0$ precisely when $1 < p < \infty$.  When $p = 1$, the solution of the mean curvature equation \eqref{mc-flow} corresponding to
a nonnegative compactly supported initial value will vanish in finite time, a property known as finite extinction; this fact
can be verified from formula \eqref{mc-soln} by analyzing what happens with each application of the median operator $M^{h}_{1}$.
When $p = \infty$, the infinite speed of propagation of equation \eqref{infty-lap-cauchy} follows from formula \eqref{infty-lap-soln} 
since each application of the midrange operator $M^{h}_{\infty}$ will enlarge the support of a compactly supported function 
by an amount $\sqrt{2h}$ in the direction of steepest descent; this reflects the one-dimensional (and thus highly degenerate) nature
of the infinity-Laplacian, something that is carefully explored and exploited in \cite{akagi:abv09}.

The relationship between the $1$-homogeneous $p$-Laplacian $\Delta^{1}_{p}$ (also known as 
the normalized or game-theoretic $p$-Laplacian, cf. \cite{kawohl:vpl11}, \cite{manfredi:amv10}, \cite{peres:tow08})
and the classical $p$-Laplacian is fairly well-known by now, but we review it briefly for completeness. For $p \in (1,\infty)$, 
the classical $p$-Laplacian $\Delta_{p}$ is the operator defined formally by 
\[
\Delta_{p} \varphi ~ := ~ \mbox{div}\left( \, |D \varphi |^{p-2} D \varphi \, \right) \ .
\]
This operator arises naturally in variational problems (\cite{heinonen:npt93}, \cite{lindqvist:npl06}), but more 
recent work (\cite{julin:npe12}, \cite{juutinen:evs01}) has shown that it is also amenable to viscosity methods. 
With such methods in mind, we define a $p$-harmonic function to be a continuous viscosity solution $u$ of 
\begin{equation} \label{pharm}
-\Delta_{p} u ~ = ~ 0 \, .
\end{equation}
Formal calculations show that 
\begin{equation} \label{decomp1}
\Delta_{p} u ~ = ~ |Du|^{p-2} \left( \, \Delta u + (p-2) \Delta_{\infty} \, \right) \, ,
\end{equation}
an identity used in \cite{juutinen:evs01}
to prove that $u$ is $p$-harmonic if and only if 
\[
-\Delta u \, - \, (p-2) \Delta_{\infty} u ~ = ~ 0 
\]
in the viscosity sense.  We also find that 
\[
\Delta_{1}u ~ = ~ \Delta u - \Delta_{\infty}u \ , 
\]
from which we have the important alternative decomposition 
\begin{equation} \label{decomp2}
\Delta_{p} u ~ = ~ |Du|^{p-2} \left( \, (p-1) \Delta u \, + \, (2-p) \Delta_{1} u \, \right) \, .
\end{equation}
Based on the work in \cite{julin:npe12} and \cite{juutinen:evs01}, it follows that $u$ is $p$-harmonic if and only if $u$ is a viscosity
solution of 
\begin{equation} \label{norm-pharm}
(p-1) \Delta u \, + \, (2-p) \Delta_{1} u ~ = ~ 0 \, .
\end{equation}

These results motivate the definition of the $1$-homogeneous $p$-Laplacian given above, and we see that $\Delta^{1}_{p}$ excludes the 
gradient-dependent factor that makes the classical $p$-Laplacian $(p-1)$-homogeneous. Thanks to this difference in homogeneity, the 
elliptic and parabolic problems related to the $1$-homogeneous $p$-Laplacian seem easier to analyze than the corresponding
problems for the classical $p$-Laplacian. 

\section{Directions for future work} \label{future}

There are many open problems related to the exponential formulas proven earlier. Of these, this final section focuses on 
two particular directions in which we are working to develop these ideas further.

\subsection{Dirichlet boundary conditions}

By considering Cauchy problems exclusively so far, we have avoided complications caused by boundaries and boundary conditions. 
Dirichlet boundary conditions, in particular, pose significant challenges for nonlinear parabolic problems; at the very least,
one expects some geometric restrictions on the boundary of the domain to be necessary in order to obtain a solution that 
attains the prescribed boundary values continuously. 

In an attempt to attack such problems, we recall the approach we applied to stationary Dirichlet problems  
in \cite{hartenstine:sfe13}:
given $h > 0$, $p \in [1,\infty]$, and a domain $\Omega \subset \R^{N}$ with nonempty boundary $\partial \Omega$, 
define the statistical operator $M^{h}_{p} \colon BUC(\overline{\Omega}) \to BUC(\overline{\Omega})$ by
\[
\left(M^{h}_{p} \varphi \right)(x) := \varphi(x) \quad \mbox{for} \quad x \in \partial \Omega 
\]
and
\[
\left(M^{h}_{p} \varphi \right)(x) := \left\{ \begin{array}{c}
\displaystyle{ (1-q) \med_{ \partial B^{h}_{x} }{ \left\{ \, \varphi \, \right\} } + q \fint_{ \partial B^{h}_{x}}{ \varphi \, dy } } \, , \quad  1 \leq p \leq 2  \\
\\
\displaystyle{ ( 1 - q ) \midrange_{ \partial B^{h}_{x} }{ \left\{ \, \varphi \, \right\} } + q \fint_{ \partial B^{h}_{x}}{ \varphi \, dy  } } \, , 
\quad p \geq 2 
\end{array} \right.
\]
for $x \in \Omega$, 
where $q = q(p,N)$ is defined by \eqref{q} and the open balls $B^{h}_{x}$ are defined by 
\[
B^{h}_{x} := B(x, r^{h}(x)) \ , 
\]
with
\[
r^{h}(x) := \left\{ \begin{array}{c}
\sqrt{2h} \quad \mbox{if} \quad \mbox{dist}(x, \partial \Omega) \geq \sqrt{2h} \, , \\
\mbox{dist}(x, \partial \Omega) \quad \mbox{otherwise}.
\end{array} \right.
\]
Note that this reduces to the definition of $M^{h}_{p}$ in Section \ref{homo} if the boundary of $\Omega$ happens to be empty.

Since $M^{h}_{p}$ is a local operator, it still satisfies the consistency condition \eqref{homo-consistency}, as well as the monotonicity, stability, and
homogeneity properties that we have been using. We would therefore like to invoke the results of \cite{barles:cas91} yet again to conclude that
\begin{equation} \label{homo-dirichlet-soln}
u(t) ~ := ~ \lim_{n \to \infty}{ \left( M^{t/n}_{p} \right)^{n} u_{0} } 
\end{equation}
is the unique bounded, continuous viscosity solution on $ \overline{ \Omega } \times [0,\infty) $ of 
\begin{equation} \label{homo-dirichlet}
\left\{ \begin{array}{c}
\displaystyle{ u_{t} \ - \ c_{p,N} \, \Delta^{1}_{p} u ~ = ~ 0 \quad \mbox{for} \quad x \in \Omega \quad 
\mbox{and} \quad t > 0 \, , } \\
\\
\displaystyle{ u(x,t) ~ = ~ u_{0}(x) \quad \mbox{for} \quad x \in \partial \Omega \quad 
\mbox{and} \quad t > 0 \, , } \\
\\
u(x,0) ~ = ~ u_{0}(x) \quad \mbox{for} \quad x \in \overline{\Omega} \ ,
\end{array} \right.
\end{equation}
where $u_{0} \in BUC( \overline{\Omega} ) $, but this will simply not work for arbitrary $p$ and $\Omega$. 
When Dirichlet conditions are imposed, we lack an appropriate comparison principle in general, and our  
understanding of the interaction of $M^{h}_{p}$ with boundary conditions is far from complete. We 
handled similar issues in \cite{hartenstine:sfe13} by using stringent definitions of subsolutions and supersolutions; 
we then found such sub- and supersolutions by requiring $\partial \Omega$ to be strictly convex.
The basic problem is to determine conditions on $\Omega$ and $p$ that will guarantee the existence of a unique continuous 
viscosity solution of \eqref{homo-dirichlet}.

Such conditions are known for bounded domains and certain values of $p$. When $p = \infty$, for instance, Juutinen and Kawohl \cite{juutinen:egi06} proved
that problem \eqref{homo-dirichlet} has a unique, bounded viscosity solution for any bounded domain $\Omega$, a surprisingly general
result made possible by the extreme degeneracy of the infinity-Laplacian. When $p=2$, problem \eqref{homo-dirichlet} has a unique bounded
solution if barriers exist at each point on $\partial \Omega$ \cite{lieberman:sop96}. When $p=1$, Sternberg and Ziemer \cite{sternberg:gmc94}
proved that \eqref{homo-dirichlet} has a unique bounded solution as long as $\partial \Omega$ is strictly convex. Under this  
assumption, Ilmanen, Sternberg and Ziemer \cite{ilmanen:esg98} analyzed the asymptotic behavior of this solution, proving that the solution 
approaches the unique function of least gradient on $\Omega$ as $t \to \infty$ if the boundary data is $C^{2}$; if the boundary data is only 
continuous, however, then this asymptotic limit is merely $1$-harmonic. 
Since $1$-harmonic functions need not be unique, one naturally wonders which one is chosen as $t \to \infty$. One should compare these
results with Juutinen's work on limits of $p$-harmonic functions as $p \to 1$ \cite{juutinen:pha05}. 

What happens for other values of $p$ needs to be investigated more thoroughly. Also, the exponential formula \eqref{homo-dirichlet-soln} 
makes sense for any $u_{0} \in BUC( \overline{ \Omega} )$; how does the function $u(t)$ defined by \eqref{homo-dirichlet-soln} behave for various 
values of $p$ and various domains $\Omega$? For example, what are the properties of this function when $p=1$ and $\Omega$ does not have a
strictly convex boundary? In this case, one expects $u(t) \in BV( \Omega )$ for $t > 0 $; when and where do discontinuities develop? 
These are just a few of the questions that need to be explored.

\subsection{Homogeneous diffusion on metric measure spaces}

Instead of working on $\R^{N}$, suppose that we have a finite-dimensional metric space $( X, d )$ equipped with a Radon measure $\mu$ 
on spheres $\partial B(x,r)$, for $x \in X$ and $r \geq 0$.  Given $p \in [1,\infty]$ 
and $h > 0$, we can define the local statistical operator $M^{h}_{p} \colon BUC( X ) \to BUC( X )$ as before, since
\[
\med_{\partial B(x,\sqrt{2h}) }{ \left\{ \, \varphi \, \right\} } \ , \quad \fint_{\partial B(x,\sqrt{2h})}{ \varphi(s) \, d\mu } \ , \quad 
\mbox{and} \quad 
\midrange_{\partial B(x,\sqrt{2h})}{ \left\{ \, \varphi \, \right\} }
\]
all make sense in this setting. The flow defined by the exponential formula \eqref{homo-cauchy-soln} thus also makes sense, 
at least \textit{a priori}, in this setting; what are its properties? Even in the simplest case when $p=2$ and the metric measure space $(X, d, \mu)$
is a Riemannian manifold without boundary, when is \eqref{homo-cauchy-soln} a formula for the solution of the heat equation on 
$X$? It seems reasonable to suspect that this holds whenever $X$ is a harmonic manifold, but the answer to this 
elementary question does not appear to be readily available in the literature. Similarly, 
could \eqref{homo-cauchy-soln} be used to derive the heat kernel explicitly on non-Euclidean spaces? When $p=1$, does 
formula \eqref{homo-cauchy-soln} provide useful insights into mean curvature flow on non-Euclidean spaces?  The many  
questions related to these operators and this exponential formula when the underlying space is not $\R^{N}$ are all open 
and intriguing.

To generalize even further, we can also consider these questions for maps from the metric measure space $(X,d,\mu)$ into
a more general metric space $(Y, \delta)$ than $\R$. To that end, let $p \in [1,\infty]$ be given, and define the $p$-mean of the continuous function 
$\varphi \colon X \to Y$ over a compact connected set $E \subset X$ by
\[
m_{p}(\varphi,E) ~ := ~ \argmin_{m \in Y}{ \left\{ \ \int_{E}{ \left( \delta( \varphi(x), m ) \right)^{p} \, d\mu(x) } \ \right\} }  \ .
\]
What happens when we replace $M^{h}_{p}$ in \eqref{homo-cauchy-soln} with the nonlinear averaging operator
\[
\left( A^{h}_{p} \varphi \right)( x ) ~ := ~ m_{p}( \, \varphi, \, \partial B(x,\sqrt{2h}) \, ) \ ?
\]
There should be interesting connections between this operator, the resulting exponential formula, and semigroup methods 
for $p$-harmonic mappings between metric spaces (cf. \cite{ambrosio:gfm08}, \cite{mayer:gfn98}, \cite{sturm:sah05}).

\bibliographystyle{siam}

\end{document}